\documentclass{amsart}
\usepackage{amssymb, amsmath, amsthm, amsfonts, amscd}
\usepackage{xcolor}
\usepackage{datetime}
\usepackage{enumerate}
\usepackage{mathabx}
\usepackage{blkarray}

\newtheorem{theorem}{Theorem}[section]
\newtheorem{proposition}[theorem]{Proposition}
\newtheorem{lemma}[theorem]{Lemma}
\newtheorem{corollary}[theorem]{Corollary}

\theoremstyle{definition}
\newtheorem{definition}[theorem]{Definition}
\newtheorem{example}[theorem]{Example}

\newtheorem{note}[theorem]{Note}

\theoremstyle{remark}
\newtheorem{remark}[theorem]{Remark}

\numberwithin{equation}{section}

\begin{document}
	
	\title[Weyl's theorem]{Weyl's theorem for commuting tuple of paranormal  and $\ast$-paranormal operators}
	
	\author{Neeru Bala}
	\address{Department of Mathematics, Indian Institute of Technology - Hyderabad, Kandi, Sangareddy, Telangana, India 502 285.}
	\email{ma16resch11001@iith.ac.in}
	
	\author{G. Ramesh}
	\address{Department of Mathematics, Indian Institute of Technology - Hyderabad, Kandi, Sangareddy, Telangana, India 502 285.}
	\email{rameshg@iith.ac.in}
	
	\date{\today}
	
	
	\begin{abstract}
	In this article, we show that a commuting pair $T=(T_1,T_2)$ of $\ast$-paranormal operators $T_1$ and $T_2$ with quasitriangular property satisfy the Weyl's theorem-I, that is $$\sigma_T(T)\setminus\sigma_{T_W}(T)=\pi_{00}(T)$$ and a commuting pair of paranormal operators satisfy Weyl's theorem-II, that is $$\sigma_T(T)\setminus\omega(T)=\pi_{00}(T),$$
	where $\sigma_T(T),\, \sigma_{T_W}(T),\,\omega(T)$ and $\pi_{00}(T)$ are the Taylor spectrum, the Taylor Weyl spectrum, the joint Weyl spectrum and the set consisting of isolated eigenvalues of $T$ with finite multiplicity, respectively.
	Moreover, we prove that Weyl's theorem-II holds for $f(T)$, where $T$ is a commuting pair of paranormal operators and $f$ is an analytic function in a neighbourhood of $\sigma_T(T)$.
	\end{abstract}
	
	\maketitle
\section{Introduction}
The aim of this article is two-fold. Firstly, we prove that a commuting pair of $\ast$-paranormal operators satisfy Weyl's theorem-I. Secondly, we show that for a commuting pair of paranormal operators $T_1$ and $T_2$, $f(T_1,T_2)$ satisfy Weyl's theorem-II for every analytic function $f$ in a neighbourhood of the Taylor spectrum of $T$.

Let $H$ denote an infinite dimensional complex Hilbert space and $\mathcal{B}(H)$ denote the space of all bounded linear operators on $H$.  If $A\in\mathcal{B}(H)$, then the spectrum of $A$ is defined by $$\sigma(A)=\{\lambda\in\mathbb{C}:A-\lambda I\text{ is not invertible in }\mathcal{B}(H)\}.$$ Recall that $A$ is said to be \textit{normal}, if $A^*A=AA^*$ and \textit{hyponormal}, if $AA^*\leq A^*A$ or equivalently $\|A^*x\|\leq\|Ax\|$, for every $x\in H$.
An operator $A\in\mathcal{B}(H)$ is said to be $\ast$\textit{-paranormal}, if
\begin{equation*}
\|A^*x\|^2\leq\|A^2x\|\|x\|,\,\forall\, x\in H
\end{equation*}
and it is said to be \textit{paranormal}, if
\begin{equation*}
\|Ax\|^2\leq\|A^2x\|\|x\|, \forall\,x\in H.
\end{equation*}
We have the following inclusion relation among these classes.
\begin{equation*}
\text{Normal}\subseteq\text{ Hyponormal }\subseteq\text{ Paranormal.}
\end{equation*}
Also Hyponormal $\subseteq\ast$-Paranormal, but the classes of $\ast$-paranormal and paranormal operators are not related to each other.

We say $A$ satisfy Weyl's theorem if $\sigma(A)\setminus\omega(A)=\pi_{00}(A)$, where $\omega(A)$ is Weyl spectrum for $A$ and $\pi_{00}(A)$ is the set of all isolated eigenvalues of $A$ with finite multiplicity. Uchiyama \cite{UCHIYAMAp} proved that a paranormal operator satisfies Weyl's theorem and later Uchiyama and Tanahashi \cite{UCHIYAMA} proved that a $\ast$-paranormal operator satisfy Weyl's theorem.

In this article, we study Weyl theorems for commuting tuple of operators.

Let $T=(T_1,T_2,\ldots T_n)$ be a commuting $n$-tuple of operators in $\mathcal{B}(H)$ and $E^n[e]=\left\{E_k^n[e]\right\}_{k=0}^n$ be an exterior algebra generated by indeterminates $e_1,e_2,\ldots e_n$, where $e_i\wedge e_j=-e_j\wedge e_i$. For a Hilbert space $H$, we denote $E^n(H)=H\otimes E^n[e]$. Define $\delta_k:E_k^n(H)\rightarrow E_{k-1}^n(H)$ by
\begin{equation*}
\delta_k(x\otimes e_{i_1}\wedge e_{i_2}\wedge\ldots e_{i_k})=\underset{j=1}{\overset{k}{\sum}}(-1)^{j-1}T_jx\otimes e_{i_1}\wedge e_{i_2}\wedge\ldots\wedge \hat{e}_{i_j}\wedge\ldots e_{i_k}.
\end{equation*}
Then the \textit{Koszul complex} $E(H,T)$ for $T$ is represented by
\begin{equation*}
0\rightarrow E_n^n(H)\xrightarrow{\delta_n}E_{n-1}^n(H)\xrightarrow{\delta_{n-1}}\cdots\rightarrow E_1^n(H)\xrightarrow{\delta_1}E_0^n(H)\rightarrow 0.
\end{equation*}
For $n=2$, the Koszul complex for $T=(T_1,T_2)$ is given by
\begin{equation*}
0\rightarrow H\xrightarrow{\delta_2}H\oplus H\xrightarrow{\delta_1} H\rightarrow 0,
\end{equation*}
where $\delta_2(x)=(-T_2x,T_1x)$ and $\delta_1(x,y)=T_1x+T_2y$, for $x,y\in H$.

The $k$-th homology group for the Koszul complex $E(H,T)$ is defined by
\begin{equation*}
H_k(E(H,T))=\text{Kernel }\delta_k/\text{Range }\delta_{k+1},\text{ for }k=0,1,\ldots n.
\end{equation*}
Then $T=(T_1,T_2,\ldots T_n)$ is said to be \textit{Taylor invertible}, if $H_k(E(H,T))=\{0\}$ and is said to be \textit{Taylor Fredholm}, if $H_k(E(H,T))$ is finite dimensional for every $k=0,1,\ldots n$. In the latter case, $\text{ind}(T)=\underset{k=1}{\overset{n}{\sum}}(-1)^k\text{dim}H_k\left(E(H,T)\right).$

The \textit{Taylor spectrum} is defined by
\begin{align*}
\sigma_T(T):=\{\lambda\in\mathbb{C}^n:T-\lambda I\text{ is not Taylor invertible}\}.
\end{align*}
An element $\lambda=(\lambda_1,\lambda_2,\ldots\lambda_n)\in\mathbb{C}^n$ is called a \textit{joint eigenvalue} of $T$, if there exist a non-zero vector $x\in H$ such that $(T_i-\lambda_i I)x=0$, for all $i=1,2,\ldots n$. We denote the set of all joint eigenvalues of $T$ by $\sigma_p(T)$ and the set of all isolated joint eigenvalues of finite multiplicity by $\pi_{00}(T)$. Note that for a single operator $A\in\mathcal{B}(H)$, the Taylor spectrum coincides with the usual spectrum of $A$.

An operator $A\in \mathcal B(H)$ is called \textit{finite rank}, if $R(A)$ is finite dimensional and $A$ is called \textit{compact}, if $A$ maps every bounded set in $H$ to a pre-compact set in $H$. We denote the space of all finite rank operators and compact operators in $\mathcal{B}(H)$ by $\mathcal{F}(H)$ and $\mathcal{K}(H)$, respectively.

For a commuting $n$-tuple of operators in $\mathcal{B}(H)$, there are two notions of Weyl spectrum. The \textit{joint Weyl spectrum},
\begin{equation*}
\begin{split}
\omega(T)=\cap\{\sigma_T(T+K):&K=(K_1,K_2,\ldots K_n),\text{ where }K_i\in\mathcal{K}(H)\\
&\text{ and }T+K\text{ is a commuting $n$-tuple}\}
\end{split}
\end{equation*}
and the \textit{Taylor Weyl spectrum},
\begin{equation*}
\sigma_{T_W}(T):=\{\lambda\in\mathbb{C}^n:T-\lambda I\text{ is not Taylor Fredholm of index 0}\}.
\end{equation*}
Subsequently, Han and Kim \cite{KIM} discussed two notions of Weyl's theorem. We say $T$ satisfy
\begin{enumerate}
	\item the \textit{Weyl's theorem-I}, if $$\sigma_T(T)\setminus\sigma_{T_W}(T)=\pi_{00}(T).$$
	\item the \textit{Weyl's theorem-II}, if $$\sigma_T(T)\setminus\omega(T)=\pi_{00}(T).$$
\end{enumerate}
  If $T_i$'s are normal operators, then $\omega(T)=\sigma_{T_W}(T)$ and consequently the Weyl's theorem-I agrees with the Weyl's theorem-II. Note that $\sigma_{T_W}(T)=\omega(T)$, if $n=1$. So Weyl's theorem-I is same as Weyl's theorem-II for $n=1$.

Cho \cite{CHOIII} proved that a commuting $n$-tuple of normal operators satisfy the Weyl's theorem-II, which is equivalent to say that it satisfies the Weyl's theorem-I. Han and Kim \cite{KIM} considered a doubly commuting $n$-tuple of hyponormal operators and showed that if the $n$-tuple has quasitriangular property, then it satisfies the Weyl's theorem-I. Relaxing the double commuting condition Chavan and Curto \cite{CHAVAN} showed that a commuting pair of hyponormal operators with quasitriangular property satisfy the Weyl's theorem-I.

In this article, we consider a commuting pair of $\ast$-paranormal operators and prove that if the commuting pair has quasitriangular property, then it satisfies the Weyl's theorem-I. Later we observe that if $(T_1,T_2)$ is a commuting pair of paranormal operators, then it satisfy Weyl's theorem-II. In this case, we have a more stronger result, which states that $f(T_1,T_2)$ also satisfy Weyl's theorem-II for every analytic function $f$ in a neighbourhood of $\sigma_T(T_1,T_2)$.

This article is divided into four sections. Section $2$ consists of some known results, which we use frequently in later sections. In section $3$ we examine the spectral properties of commuting tuple of $\ast$-paranormal operators and prove the Weyl's theorem for a commuting pair of $\ast$-paranormal operators with the quasitriangular property. In section $4$ we discuss the Weyl's theorem-II for $f(T_1,T_2)$, where $(T_1,T_2)$ is a commuting pair of paranormal operators and $f$ is an analytic function in a neighbourhood of $\sigma_T(T_1,T_2)$.

\section{Preliminaries}
In this section, we quote some relevant material from the spectral theory of commuting $n$-tuple of operators in $\mathcal{B}(H).$

For $A\in\mathcal{B}(H)$, the null space and range space of $A$ are denoted by $N(A)$ and $R(A)$, respectively.
\begin{definition}
	A commuting $n$-tuple $T=(T_1,T_2,\ldots T_n)$ in $\mathcal{B}(H)$ has \textit{quasitriangular property}, if $\text{dim}\underset{i=1}{\overset{n}{\cap}}N(T_i-\lambda_iI)\geq\text{dim}\underset{i=1}{\overset{n}{\cap}}N(T_i-\lambda_i I)^*$, for every $\lambda=(\lambda_1,\lambda_2,\ldots\lambda_n)\in\sigma_T(T)$.
\end{definition}
This is a generalization of upper triangular matrices. Note that a commuting $n$-tuple of normal operators have quasitriangular property.

Vasilescu \cite{VASILESCU}, gave the following characterization for the Taylor spectrum of a commuting pair of operators. This is very useful from the application point of view.
\begin{theorem}\cite{VASILESCU}
	Let $T=(T_1,T_2)$ be a commuting pair in $\mathcal{B}(H)$. Then $T$ is Taylor invertible if and only if the operator
	\begin{equation}\label{ Vas matrix representation }
	\hat{T}=
	\begin{bmatrix}
	T_1&-T_2^*\\
	T_2&T_1^*
	\end{bmatrix}
\end{equation}
	is invertible on $H\oplus H$.
\end{theorem}
The following result yields that the Taylor joint spectrum has projection property.
\begin{lemma}\cite[Lemma 3.1]{TAYLOR}
	Let $T=(T_1,T_2,\ldots T_n)$ be a commuting $n$-tuple of operators in $\mathcal{B}(H)$. If $\pi_i:\mathbb{C}^n\rightarrow\mathbb{C}$ is the projection on the i-th coordinate, then $\pi_i(\sigma_T(T_1,T_2,\ldots T_n))=\sigma(T_i)$ for $1\leq i\leq n$.
\end{lemma}
Suppose $T=(T_1,T_2,\ldots T_n)$ is a commuting $n$-tuple of operators in $\mathcal{B}(H)$ and $M\subseteq H$ is a subspace of $H$. Then $M$ is said to be invariant under $T$, if $M$ is invariant under $T_i$ for every $1\leq i\leq n$.
\begin{theorem}\cite[Application 5.24, Page 59]{CURTO}\label{thm disjoint spectra}
	Let $T=(T_1,T_2,\ldots T_n)$ be a commuting $n$-tuple of operators in $\mathcal{B}(H)$ and $\sigma_T(T)=K_1\cup K_2$, where $K_1$ and $K_2$ are nonempty disjoint compact subsets of $\sigma_T(T)$. Then there exist $T$-invariant subspaces $M_1$ and $M_2$ such that $M_1\cap M_2=\{0\}, M_1+ M_2=H$ and
	$\sigma_T(T|_{M_i})\subseteq K_i$, for $i=1,2$.
\end{theorem}
\begin{definition}\cite[Definition 2.2]{CURTO HERRE}
	Let $T=(T_1,T_2,\ldots T_n)$ and $S=(S_1,S_2,\ldots S_n)$ be two $n$-tuple of operators on $H$. Then $T$ is said to be jointly similar to $S$ if there exist an invertible operator $W\in\mathcal{B}(H)$ such that $T_j=WS_jW^{-1}$ for $j=1,2,\ldots n$.
\end{definition}
\begin{corollary}\cite[Application 5.24(continued), Page 60]{CURTO}\label{matrix rep}
	Let $T$ be as defined in Theorem \ref{thm disjoint spectra}. Then $T$ is jointly similar to
	\[
	\begin{bmatrix}
	T|_{M_1}&0\\
	0&A
	\end{bmatrix},
	\]
	where $A\in\mathcal{B}(M_1^{\perp})$, $\sigma_T(T|_{M_1})=K_1$ and $\sigma_T(A)=K_2$.
\end{corollary}

Suppose $T=(T_1,T_2,\ldots T_n)$ be an $n$-tuple of commuting operators in $\mathcal{B}(H)$. Then we can associate $T$ with $R(T):=\delta_T+\delta_T^*$, where $\delta_T=T_1S_1+T_2S_2+\cdots T_nS_n$ and $S_j\zeta:=e_j\wedge \zeta$ for $j=1,\ldots n$, $\zeta\in E^n[e].$ For $(z_1,z_2,\ldots z_n)\in\mathbb{C}^n$,
$$\bar{\partial}=\frac{\partial}{\partial\bar{z}_1}d\bar{z}_1+\frac{\partial}{\partial\bar{z}_2}d\bar{z}_2+\cdots+\frac{\partial}{\partial\bar{z}_n}d\bar{z}_n.$$
 We write $H(\sigma_T(T))$ for the set of all analytic functions in a neighbourhood of $\sigma_T(T)$. We denote the set of all $X$-valued analytic functions in $U$ by $H(U,X)$. Let $f\in H(\sigma_T(T))$ and $U$ be a neighbourhood of $\sigma_T(T)$ such that $f$ is analytic in $U$. It is possible to choose a compact neighbourhood $\Delta$ of $\sigma_T(T)$ such that $\Delta\subset U$ and boundary $\partial\Delta$ is a smooth surface. Then the operator $f(T)$ is defined by
\begin{equation*}
f(T):=\frac{1}{(2\pi i)^n}\int_{\partial\Delta}f(w)M(w-T)dw,
\end{equation*}
where
$$M(w-T)\eta:=(-1)^{n-1}PH(w-T)T\eta\text{ for }\eta\in H(U,X)\otimes E^n[e]\text{ with }\bar{\partial}\eta=0,$$ $P$ is the projection from $H(U,X)\otimes E^n[e,d\bar{z}]$ onto $H(U,X)\otimes E^n[e]$ and $$H(z-T)\eta(z):=\underset{j=0}{\overset{p-1}{\sum}}\underset{k=0}{\overset{j}{\sum}}(-1)^k(R(z-T))^{-1}(\bar{\partial}(R(z-T))^{-1})^k\eta_{j-k}(z),$$ for $\eta\in H(U,X)\otimes E_p^n[e,d\bar{z}]$.
\begin{theorem}\cite[Theorem 5.19, Page 56]{CURTO}
	Let $T=(T_1,T_2,\ldots T_n)$ be a commuting $n$-tuple in $\mathcal{B}(H)$ and $f\in H(\sigma_T(T))$. Then $\sigma_T(f(T))=f(\sigma_T(T))$.
\end{theorem}
For more details about analytic functional calculus, we refer to \cite{CURTO,VASILESCU2}.

For a commuting $n$-tuple of operators the Taylor Weyl spectrum and the joint Weyl spectrum are related as follows.
\begin{lemma}\cite[Lemma 2]{KIM}\label{lemma joint weyl spectrum}
	If $T=(T_1,T_2,\ldots T_n)$ is a commuting $n$-tuple of operators in $\mathcal{B}(H)$, then $\sigma_{T_W}(T)\subseteq\omega(T)$.
\end{lemma}
The above inclusion is strict. For example, consider the right shift operator $R$ on $l_2(\mathbb{N})$. Then $\omega(R,R)=\overline{D(0,1)}\times\overline{D(0,1)}\supsetneq\sigma_{T_W}(R,R)=\partial D(0,1)\times\partial D(0,1),$ where $D(0,1)=\{\lambda\in\mathbb{C}:|\lambda|<1\}$.

 \begin{theorem}\cite[Theorem 2.8]{DASH}\label{thm dash}
 	Let $T=(T_1,T_2,\ldots T_n)$ be an $n$-tuple of operators on $H$. Then $\sigma_T(T)=\sigma_{T_e}(T)\cup\sigma_p(T)\cup\sigma_p(T^*)^*$, where $$\sigma_{T_e}(T)=\{\lambda\in\mathbb{C}^n:T-\lambda I\text{ is not Taylor Fredholm}\}.$$
 \end{theorem}

\section{$\ast$-Paranormal operators}
This section is dedicated to the study of Weyl's theorem for a commuting tuple of $\ast$-paranormal operators.

 A linear operator $T\in\mathcal{B}(H)$ is called \textit{isoloid}, if every isolated point of $\sigma_T(T)$ is an eigenvalue. From \cite[Corollary 1]{UCHIYAMA} and \cite[Lemma 3.4]{UCHIYAMAp}, it is well known that every $\ast$-paranormal and paranormal operator are isoloid. Now we extend these results to commuting $n$-tuple of $\ast$-paranormal and paranormal operators.
\begin{lemma}\label{Lemma *isoloid}
	Let $T=(T_1,T_2,\ldots T_n)$ be a commuting $n$-tuple of $\ast$-paranormal (or paranormal) operators. If $\lambda\in\mathbb{C}^n$ is an isolated point of $\sigma_T(T)$, then $\lambda\in\sigma_p(T)$.
\end{lemma}
\begin{proof}
	Let $\lambda=(\lambda_1,\lambda_2\ldots\lambda_n)$ be an isolated element of $\sigma_T(T)$. By Theorem \ref{thm disjoint spectra}, there exist $T$-invariant subspaces $M_1$ and $M_2$ such that $\sigma_T(T|_{M_1})=\{\lambda\}$ and $\sigma_T(T|_{M_2})=\sigma_T(T)\setminus\{\lambda\}$. Clearly $\underset{i=1}{\overset{n}{\cap}}N(T_i-\lambda_i I)\subseteq M_1$. We claim that  $\underset{i=1}{\overset{n}{\cap}}N(T_i-\lambda_i I)= M_1$.
	
	As $M_1$ is $T$-invariant subspace, thus $T|_{M_1}=(T_1|_{M_1},\ldots T_n|_{M_1})$ is a commuting $n$-tuple of $\ast$-paranormal (or paranormal) operators. By projection property of Taylor joint spectrum, $\sigma(T_i|_{M_1})=\{\lambda_i\}$. As a result of \cite[Corollary 1]{UCHIYAMA} and \cite[Lemma 3.4]{UCHIYAMAp}, $T_i|_{M_1}-\lambda_iI_{M_1}=0$, for every $1\leq i\leq n$. This implies $M_1\subseteq\underset{i=1}{\overset{n}{\cap}}N(T_i-\lambda_i I)$ and consequently $M_1=\underset{i=1}{\overset{n}{\cap}}N(T_i-\lambda_i I)$. Hence $\lambda\in\sigma_p(T).\qedhere$
\end{proof}
\begin{proposition}\label{isolated point of *para}
Let $T=(T_1,T_2,\ldots T_n)$ be a commuting $n$-tuple of $\ast$-paranormal operators in $\mathcal{B}(H)$. Then $\pi_{00}(T)\subseteq\sigma_T(T)\setminus\sigma_{T_W}(T)$.
\end{proposition}
\begin{proof}
Let $\lambda=(\lambda_1,\lambda_2\ldots\lambda_n)\in\pi_{00}(T)$. By Theorem \ref{thm disjoint spectra}, there exist $T$-invariant subspaces $M_1$ and $M_2$ such that $\sigma_T(T|_{M_1})=\{\lambda\}$ and $\sigma_T(T|_{M_2})=\sigma_T(T)\setminus\{\lambda\}$. From the proof of Lemma \ref{Lemma *isoloid}, it is clear that $M_1=\underset{i=1}{\overset{n}{\cap}}N(T_i-\lambda_i I)$. By \cite[Lemma 4]{UCHIYAMA}, we have $N(T_i-\lambda_i I)=N((T_i-\lambda_iI)^*)$. Hence $M_1$ is a reducing subspace for $T$ and $T|_{M_1^{\perp}}$ is a commuting $n$-tuple of $\ast$-paranormal operators.

If $\lambda\in\sigma_T(T|_{M_1^{\perp}})$, then $\lambda\in\sigma_p(T|_{M_1^{\perp}})$, by Lemma \ref{Lemma *isoloid}. But this contradicts the fact that $M_1=\underset{i=1}{\overset{n}{\cap}}N(T_i-\lambda_i I)$. Hence $\lambda\notin\sigma_T(T|_{M_1^{\perp}})$. If $\lambda\ne 0$, then $\lambda\notin\sigma_T(T-TP_{M_1})$ and if $\lambda=0$, then $\lambda\notin\sigma_T(T+P_{M_1})$.

As $M_1$ is finite dimensional, we have $TP_{M_1}$ and $P_{M_1}$ are finite rank operators. So $\lambda\notin\omega(T)$ and as a consequence of Lemma \ref{lemma joint weyl spectrum}, $\lambda$$\notin$$\sigma_{T_W}(T)$. Hence $\lambda\in\sigma_T(T)\setminus\sigma_{T_W}(T).$
\end{proof}
By Lemma \ref{lemma joint weyl spectrum}, we know that $\sigma_{T_W}(T)\subseteq\omega(T)$ and in general the equality need not hold. Here we show that equality holds for a subclass of $\ast$-paranormal operators.
\begin{proposition}\label{prop*paranormal}
	Let $T=(T_1,T_2)$ be a commuting pair of $\ast$-paranormal operators in $\mathcal{B}(H)$. If $T$ has quasitriangular propetry, then $\omega(T)=\sigma_{T_W}(T)$.
\end{proposition}
\begin{proof}
In view of Lemma \ref{lemma joint weyl spectrum}, it is enough to show that $\omega(T)\subseteq\sigma_{T_W}(T)$. Let $\lambda\notin\sigma_{T_W}(T)$. To show $\lambda\notin\omega(T)$, we need to find a pair of compact operators $K=(K_1.K_2)$ such that $T+K-\lambda I$ is not Taylor invertible.

 As a consequence of \cite[Lemma 4]{UCHIYAMA} and quasitriangular property of $T$, we have $$N(T_1-\lambda_1I)\cap N(T_2-\lambda_2I)= N((T_1-\lambda_1I)^*)\cap N((T_2-\lambda_2I)^*).$$

Let $K$ be the projection onto $\underset{i=1}{\overset{2}{\cap}}N(T_i-\lambda_iI)$. Then $K$ is a finite-rank operator and $T+K-\lambda I:=(T_1+K-\lambda_1 I,T_2+K-\lambda_2 I)$. Now, our goal is to show that $T+K-\lambda I$ is a commuting pair of operators in $\mathcal{B}(H)$ and $0\notin\sigma_T(T+K-\lambda I)$.

We claim that $T_iK=KT_i=\lambda_iK$. To prove our claim, let $x\in H$. Then $x=x_1+x_2$, where $x_1\in\underset{i=1}{\overset{2}{\cap}}N(T_i-\lambda_i I)$ and $x_2\in\left(\underset{i=1}{\overset{2}{\cap}}N(T_i-\lambda_i I)\right)^{\perp}.$ We have $$T_iKx=T_iKx_1=T_ix_1=\lambda_ix_1=\lambda_iKx.$$ As $\underset{i=1}{\overset{2}{\cap}}N(T_i-\lambda_iI)$ is a reducing subspace for $T$, we have $KT_ix=\lambda_ix_1=\lambda_iKx$. Now, it can be easily verified that $$(T_1+K-\lambda_1I)(T_2+K-\lambda_2I)=(T_2+K-\lambda_2I)(T_1+K-\lambda_1I).$$

By \cite[Theorem 2, 3]{CURTOfredholm}, we know that $T+K-\lambda I$ is Fredholm of $\text{ind}(T+K-\lambda I)=\text{ind}(T-\lambda I)=0$. Let $E(H,(T+K-\lambda I))$ be the associated chain complex of $T+K-\lambda I$. Our next objective is to show that $T+K-\lambda I$ is Taylor invertible.

Let $x\in\underset{i=1}{\overset{2}{\cap}}N(T_i+K-\lambda_i I)$ and $x=x_1+x_2$, where $x_1\in\underset{i=1}{\overset{2}{\cap}}N(T_i-\lambda_i I)$ and $x_2\in\left(\underset{i=1}{\overset{2}{\cap}}N(T_i-\lambda_i I)\right)^{\perp}.$ As $\underset{i=1}{\overset{2}{\cap}}N(T_i-\lambda_iI)$ is a reducing subspace for $T$, we have the following equality.
\begin{align*}
0=\langle(T_i+K-\lambda_iI)x,x_1\rangle=\langle Kx_1,x_1\rangle
=\|Kx_1\|^2.
\end{align*}
It follows that $Kx_1=0$ and consequently $x_1\in\left(\underset{i=1}{\overset{2}{\cap}}N(T_i-\lambda_iI)\right)\cap\left(\underset{i=1}{\overset{2}{\cap}}N(T_i-\lambda_iI)\right)^{\perp}=\{0\}$. For $i=1,2$, we have
\begin{align*}
(T_i-\lambda_iI)(x_2)=&(T_i-\lambda_iI)(x_1+x_2)\\
=&(T_i-\lambda_iI+K)(x)\\
=&0.
\end{align*}
 This implies that $x_2\in\left(\underset{i=1}{\overset{2}{\cap}}N(T_i-\lambda_iI)\right)\cap\left(\underset{i=1}{\overset{2}{\cap}}N(T_i-\lambda_iI)\right)^{\perp}=\{0\}$ and thus $x=0$. Hence $H_0\left(E(H,(T+K-\lambda I))\right)=\{0\}.$ Similarly, we get $\underset{i=1}{\overset{2}{\cap}}N(T_i+K-\lambda_iI)^*=\{0\}$. As we know that
$$\text{dim}H_2\left(E(H,(T+K-\lambda I))\right)=\text{dim} N(T_i+K-\lambda_i I)^*=0,$$ this implies $H_2\left(E(H,(T+K-\lambda I))\right)=\{0\}$.
Since $\text{ind}(T+K-\lambda I)=0$, it follows that $H_1\left(E(H,(T+K-\lambda I))\right)=\{0\}$. So $0\notin\sigma_T\left(T+K-\lambda I\right)$ and hence $\lambda\notin\omega(T).\qedhere$
\end{proof}
We extend Chavan and Curto \cite{CHAVAN} result to $\ast$-paranormal operators.
\begin{theorem}\label{Weylthm *paranormal}
	Let $T=(T_1,T_2)$ be a commuting tuple of $\ast$-paranormal operators in $\mathcal{B}(H)$. If $T$ has quasitriangular property, then $T$ satisfies Weyl's theorem-I, that is $\sigma_T(T)\setminus\sigma_{T_W}(T)=\pi_{00}(T)$.
\end{theorem}
\begin{proof}
	By Proposition \ref{isolated point of *para}, we have $\pi_{00}(T)\subseteq\sigma_T(T)\setminus\sigma_{T_W}(T)$. To prove the reverse inclusion, we consider $\lambda\in\sigma_T(T)\setminus\sigma_{T_W}(T)$. Thus $T-\lambda I$ is Fredholm of index zero. Using \cite[Corollary 3.6]{CURTOfredholm}, we can conclude that $\underset{i=1}{\overset{2}{\sum}}(T_i-\lambda_i I)^*(T_i-\lambda_iI)$ is also Fredholm. If $\lambda$ is isolated, then by Lemma \ref{Lemma *isoloid}, $\lambda\in\pi_{00}(T)$ and the conclusion holds.
	
	On the other hand, if $\lambda$ is not isolated point of $\sigma_T(T)$, then there exist a sequence $(\lambda_n)\subseteq\sigma_T(T)$ such that $\lambda_n\rightarrow\lambda$. As $\sigma_{T_W}(T)$ is closed, we have $(\lambda_n)\subseteq\sigma_T(T)\setminus\sigma_{T_W}(T)$, otherwise $\lambda\in\sigma_{T_W}(T)$. By Theorem \ref{thm dash}, $\sigma_T(T)\setminus\sigma_{T_e}(T)\subseteq\sigma_p(T)\cup\sigma_p(T^*)^*\subseteq\sigma_p(T^*)^*$. Thus $\overline{\lambda_n}\in\sigma_p(T^*)$ for every $n\in\mathbb{N}$. So there exist an orthonormal sequence $(x_n)$ such that $(T_i-\lambda_{n_i})^*x_n=0$ for $i=1,2$. By quasitriangular property, it follows that $(T_i-\lambda_{n_i})x_n=0$ for $i=1,2$. Let $M=\text{max}\{\|(T_1-\lambda_1I)^*\|,\|(T_2-\lambda_2I)^*\|\}$. Then
	\begin{align*}
	\left\|\underset{i=1}{\overset{2}{\sum}}(T_i-\lambda_iI)^*(T_i-\lambda_iI)x_n\right\|&\leq\underset{i=1}{\overset{2}{\sum}}\|(T_i-\lambda_iI)^*(T_i-\lambda_iI)x_n\|\\
	&\leq M\underset{i=1}{\overset{2}{\sum}}\|(T_i-\lambda_iI)x_n\|\\
&\leq M\underset{i=1}{\overset{2}{\sum}}\|(T_i-\lambda_{n_i}I)x_n\|+M\underset{i=1}{\overset{2}{\sum}}|\lambda_{n_i}-\lambda_i|\\
	&=M\underset{i=1}{\overset{2}{\sum}}|\lambda_{n_i}-\lambda_i|\rightarrow 0\text{ as }n\rightarrow \infty.
	\end{align*}
	 Thus $\underset{i=1}{\overset{2}{\sum}}(T_i-\lambda_i I)^*(T_i-\lambda_iI)$ is not Fredholm. This is a contradiction to the assumption that $\lambda\in\sigma_T(T)\setminus\sigma_{T_W}(T)$. Hence $\lambda$ is isolated point of $\sigma_T(T)$ and consequently $\lambda\in\pi_{00}(T).\qedhere$
\end{proof}
If the quasitriangular property is dropped, then Theorem \ref{Weylthm *paranormal} need not hold. The following example illustrates this.
\begin{example}
Let $T:l_2(\mathbb{N})\rightarrow l_2(\mathbb{N})$ be a weighted shift operator defined by
\begin{align*}
T(e_n)=
\begin{cases}
\sqrt{2}e_2,&\text{ if }n=1,\\
e_3,&\text{ if }n=2,\\
2e_{n+1}&\text{ if }n\geq 3.
\end{cases}
\end{align*}
Here $T$ is $\ast$-paranormal operator, but not a paranormal operator. We know that $\sigma_T(T)=\overline{D(0,2)}$, where $D(0,2)=\{\lambda\in\mathbb{C}:|\lambda|<2\}.$

Using projection property, we know that $\sigma_T(T,0)=\overline{D(0,2)}\times\{0\}.$ As $\sigma_p(T)=\emptyset,$ we have $\pi_{00}(T,0)=\emptyset$. It is also easy to see that $\omega(T,0)=\overline{D(0,2)}\times\{0\}$ and $\sigma_{T_W}(T,0)=\partial D(0,2)\times\{0\}.$

If we consider the commuting tuple $(T,0)$ of $\ast$-paranormal operators, then it does not satisfy the Weyl's theorem-I, that is $\sigma_T(T,0)\setminus\sigma_{T_W}(T,0)\neq\pi_{00}(T,0)$. Observe that $(T,0)$ does not have quasitriangular property.
\end{example}
\section{Paranormal operators}
In this section, our main aim is to show that for every analytic function $f$ in a neighbourhood of the spectrum of a commuting pair of paranormal operators $(T_1,T_2)$, $f(T_1,T_2)$ satisfy Weyl's theorem-II.

First, we show that a commuting pair of paranormal operators satisfy the Weyl's theorem-II.
\begin{theorem}\label{Weyls thm}
	Let $T=(T_1,T_2)$ be a commuting pair of paranormal operators in $\mathcal{B}(H)$. Then $\sigma_T(T)\setminus\omega(T)=\pi_{00}(T)$.
\end{theorem}
\begin{proof}
	Let $\lambda\in\pi_{00}(T)$. Then there exist $T$-invariant subspaces $M_1$ and $M_2$ such that $\sigma_T(T|_{M_1})=\{\lambda\}$ and $\sigma_T(T|_{M_2})=\sigma_T(T)\setminus\{\lambda\}$. By Lemma \ref{Lemma *isoloid}, $M_1=\underset{i=1}{\overset{2}{\cap}}N(T_i-\lambda_i I)$.
	
	If $\lambda\ne (0,0)$, then
	\[
	T-TP_{M_1}=
	\begin{bmatrix}
	0&0\\
	0& T|_{M_2}
	\end{bmatrix}.
	\]
	Thus $\lambda\notin\sigma_T(T-TP_{M_1})$. Here $TP_{M_1}$ is a finite rank operator, hence $\lambda\notin\omega(T)$.
	
	If $\lambda=(0,0)$, then
	\[
	T+P_{M_1}
	=\begin{bmatrix}
	I_{M_1}&0\\
	0&T|_{M_2}
	\end{bmatrix}.
	\]
It is easy to see that $(0,0)\notin\sigma_T(T+P_{M_1})$. Hence $\lambda=(0,0)\notin\omega(T)$. In both the cases $\lambda\in\sigma_T(T)\setminus\omega(T)$.

Conversely, let $\lambda\in\sigma_T(T)\setminus\omega(T)$. By definition of $\omega(T)$, there exist a pair of compact operators $K_1$ and $K_2$ such that if $K=(K_1,K_2)$, then $T+K$ is a commuting tuple and $\lambda\notin\sigma_T(T+K)$.

 If $\lambda$ is isolated point of $\sigma(T)$, then $\lambda\in\sigma_p(T)$, by Lemma \ref{Lemma *isoloid}. We have to show that $\underset{i=1}{\overset{2}{\cap}}N(T_i-\lambda_iI)$ is finite dimensional. On the contrary, if $\underset{i=1}{\overset{2}{\cap}}N(T_i-\lambda_iI)$ is infinite dimensional, then there exist an orthonormal sequence $(x_n)\subseteq H$ such that $(T_i-\lambda_iI)x_n=0$, for $i=1,2$. As $(x_n)$ converges to zero weakly, we get $(K_1x_n, K_2x_n)$ converges to zero strongly. Consider
 \[
 \widehat{(T+K-\lambda I)}(x_{n},0)=
 \begin{bmatrix}
 T_1+K_1-\lambda_1I&(T_2+K_2-\lambda_2I)^*\\
 T_2+K_2-\lambda_2I&(T_1+K_1-\lambda_1I)^*
 \end{bmatrix}
 \begin{bmatrix}
 x_{n}\\
 0
 \end{bmatrix}
 =
 \begin{bmatrix}
 K_1x_{n}\\
 K_2x_{n}
 \end{bmatrix},
 \]
 where $~~~\widehat{}~~~$ is defined as in Equation \ref{ Vas matrix representation }.
 Then
 \begin{align}\label{eqnmat}
 \|\widehat{(T+K-\lambda I)}(x_{n},0)\|=\|(K_1x_{n},K_2x_{n})\|.
 \end{align}
 The right hand side of Equation \ref{eqnmat} converges to zero as $n\rightarrow\infty$. Thus $(x_{n},0)$ converges to $(0,0)$, which contradicts the fact that $(x_{n})$ is an orthonormal sequence. Hence $\underset{i=1}{\overset{2}{\cap}}N(T_i-\lambda_i I)$ is finite dimensional.

 If $\lambda$ is not an isolated point of $\sigma_T(T)$, then there exist a sequence $(\lambda_n)\subseteq\sigma_T(T)$ such that $(\lambda_n)$ convereges to $\lambda$ as $n\rightarrow\infty$. Using the fact that $\omega(T)$ is closed, we assume that $(\lambda_n)\subseteq\sigma_T(T)\setminus\omega(T)\subseteq\sigma_p(T)\cup\sigma_p(T^*)^*$, by Theorem \ref{thm dash}.

 If $(\lambda_n)\subseteq\sigma_p(T)$, then there exist an orthonormal sequence $(y_n)$ such that $(T_i-\lambda_{n_i})y_n=0$ for $i=1,2$. Consider the sequence $(y_n,0)\subseteq H\oplus H$. Now using the compactness of $K_1$ and $K_2$ we know that $(K_1y_{n},K_2y_{n})$ converges to $(0,0)$. It is easy to see that
 \begin{align*}
 \|\widehat{(T+K-\lambda_{n} I)}(y_{n},0)\|\rightarrow 0\text{ as }n\rightarrow\infty.
 \end{align*}
 Further, we have
 \begin{align*}
 \|\widehat{(T+K-\lambda I)}(y_{n},0)\|\leq\|\widehat{(T+K-\lambda_{n} I)}(y_{n},0)\|+|\lambda_{n}-\lambda|.
 \end{align*}
 The right hand side of above equation goes to zero as $n\rightarrow\infty$.
 Again we arrives at a contradiction as $(y_{n},0)$ converges to $(0,0)$ as $n\rightarrow\infty$.

 Similarly, if $(\lambda_n)\subseteq\sigma_p(T^*)^*$, then there exist an orthonormal sequence $(w_n)$ such that $(T_i-\lambda_{n_i}I)^*w_n=0$, for $i=1,2$. On the similar lines we can show that $(0,w_n)$ converges to $(0,0)$ as $n\rightarrow 0$, which leads to a contradiction. If $(\lambda_n)$ is not contained in either of the sets $\sigma_p(T)$ and $\sigma_p(T^*)$, then we get a subsequence of $(\lambda_n)$, which is contained in one of the sets. Further we deal with the subsequence and following the same argument as above we end up with a contradiction. Hence $\lambda\in\pi_{00}(T).\qedhere$
\end{proof}
\begin{remark}
 Weyl's theorem-I need not hold for paranormal operators. Consider the right shift operator $R:\ell_2(\mathbb{N})\rightarrow \ell_2(\mathbb{N})$, defined by $R(x_1,x_2,\ldots)=(0,x_1,x_2,\ldots)$, for all $(x_n)\in\ell_2(\mathbb{N})$. Then
 \begin{align*}
 \sigma_T(R,R)=&\overline{D(0,1)}\times\overline{D(0,1)},\\
 \omega(R,R)=&\overline{D(0,1)}\times\overline{D(0,1)},\\
 \sigma_p(R,R)=&\emptyset\\
 \sigma_{T_W}(R,R)=&\partial D(0,1)\times\partial D(0,1).
 \end{align*}
 Clearly $(R,R)$ satisfies Weyl's theorem-II but not Weyl's theorem-I.	
\end{remark}

The rest of the section is devoted to the study of spectral properties of analytic functional calculus for a pair of paranormal operators.
\begin{lemma}\label{iso eigenvalue}
	Let $T=(T_1,T_2)$ be a commuting pair of paranormal operators in $\mathcal{B}(H)$. If $\delta=(\delta_1,\delta_2)$ is an isolated point of $\sigma_T(T)$, then $f(\delta)$ is an isolated joint eigenvalue of $f(T)$, for every $f\in H(\sigma_T(T))$.
\end{lemma}
\begin{proof}
	Lemma \ref{Lemma *isoloid} implies that $\delta$ is a joint eigenvalue of $T$.
	By Corollary \ref{matrix rep}, we know that $T$ is similar to the following matrix;
	\[
	\begin{bmatrix}
		T|_M&0\\
	0& A
	\end{bmatrix},\text{ for some }A\in\mathcal{B}(M^{\perp}),
	\]
	 where $M$ is an invariant subspace of $T,\, \sigma_T(T|_M)=\{\delta\}$ and $\sigma_T(A)=\sigma_T(T)\setminus\{\delta\}$. Lemma \ref{Lemma *isoloid} implies that $M=N(T_1-\delta_1I)\cap N(T_2-\delta_2 I)$. As $f(T|_M)=f(\delta)I_M$, it follows that $f(T)$ is similar to
	 \[
	 \begin{bmatrix}
	 f(T|_M)&0\\
	 0&f(A)
	 \end{bmatrix}
	 =
	 \begin{bmatrix}
	 f(\delta) I_M&0\\
	 0&f(A)
	 \end{bmatrix}.
	 \]
Thus $f(\delta)$ is a joint eigenvalue of $f(T)$ and dim$N(f(T)-f(\delta)I)=\text{dim}M$. Since $\delta\notin\sigma_T(A)$, this gives $f(\delta)\notin f(\sigma_T(A))=\sigma_T(f(A))$. Hence $f(\delta)$ is an isolated joint eigenvalue of $f(T)$ with the same multiplicity as $\delta$ in $\sigma_T(T)$.
\end{proof}
\begin{remark}\label{rem eigenvalues}
	Let $T=(T_1,T_2)$ be as defined in Lemma \ref{iso eigenvalue}. If $\delta$ is an isolated joint eigenvalue of $T$, then $f(\delta)$ is an isolated joint eigenvalue of $f(T)$. Moreover $\text{dim}\left(N(T-\delta I) \right)=\text{dim}\left(N(f(T)-f(\delta)I)\right)$.
\end{remark}
\begin{lemma}\label{lemma eigenvalue}
	Let $T=(T_1,T_2)$ be a commuting pair of paranormal operators in $\mathcal{B}(H)$. If $\delta$ is a joint eigenvalue of $T$ then $f(\delta)$ is a joint eigenvalue of $f(T)$ for every $f\in H(\sigma_T(T))$.
\end{lemma}
\begin{proof}
	If $\delta$ is an isolated joint eigenvalue of $T$, then $f(\delta)$ is also a joint eigenvalue of $f(T)$, by Lemma \ref{iso eigenvalue}.
	
	On the other hand, assume that $\delta$ is not an isolated point of $\sigma_T(T)$. Without loss of generality, we assume that $\sigma_T(T)$ is disjoint union of $K_1,\,K_2$ and $\delta\in K_1$. Then Theorem \ref{thm disjoint spectra} implies that there exist two $T$-invariant subspaces $M,\,N\subseteq H$ such that $\sigma_T(T|_M)=K_1$ and $\sigma_T(T|_N)=K_2$. Let $U_1,\,U_2$ be two open sets in $\mathbb{C}^2$ containing $K_1,\,K_2$, respectively and $f$ is an analytic function in $U_1\cup U_2$. We observe that $(\delta_1,\delta_2)$ is a zero of $f(z_1,z_2)-f(\delta_1,\delta_2)$. The folllowing two cases exhaust all the possiblities.
	
	Case $(1)$: There exist a neighbourhood $V$ of $(\delta_1,\delta_2)$ such that $f(z_1,z_2)-f(\delta_1,\delta_2)=0$ for all $(z_1,z_2)\in V$:\\ By \cite[Theorem 3.7, Page 78]{CONWAY}, $f(z_1,z_2)-f(\delta_1,\delta_2)=0$ for all $(z_1,z_2)\in U_1$. Consequently $f(T|_M)-f(\delta)I_M=0$. Hence $f(\delta)$ is a joint eigenvalue of $f(T)$.
	
	Case $(2)$: There exist a neighbourhood $V_2$ of $\delta_2$ such that $f(\delta_1,z_2)-f(\delta_1,\delta_2)\ne 0$ for all $z_2\in V_2$:\\ By \cite[Theorem 1.2.5, Page 18]{GILBERT} we have
\begin{align*}
f(z_1,z_2)-f(\delta_1,\delta_2)=h(z_1,z_2)[&(z_2-\delta_2)^m+(z_2-\delta_2)^{m-1}(z_1-\delta_1)g_1(z_1)+\ldots\\
&+(z_1-\delta_1)g_m(z_1)],
\end{align*}
	 for all $(z_1,z_2)\in U_1$ and for some $m\in\mathbb{N}$. Here $g_i(z_1)$ is an analytic function in a neighbourhood of $\delta_1$ and $h(z_1,z_2)$ is analytic function in $U_1$ such that $h(\delta_1,\delta_2)\ne 0$. Hence
	\begin{align*}
		f(T|_M)-f(\delta)I_M=&h(T_1|_M,T_2|_M)[(T_2|_M-\delta_2I_M)^m+(T_2|_M-\delta_2I_M)^{m-1}\\
		&(T_1|_M-\delta_1I_M)g_1(T_1|_M)+\ldots+(T_1|_M-\delta_1I_M)g_{m}(T_1|_M) ].
	\end{align*}
As $(\delta_1,\delta_2)\in\sigma_p(T_1|_M,T_2|_M)$, there exist some $x\in H$ such that $T_1x=\delta_1x$ and $T_2x=\delta_2x$. Thus we have $(f(T|_M)-f(\delta)I_M)x=0$. Hence $f(\delta)$ is a joint eigenvalue of $f(T)$.
\end{proof}
\begin{note}
	Let $T=(T_1,T_2)$ be as defined in Lemma \ref{lemma eigenvalue}. It is clear from the proofs of Lemma \ref{iso eigenvalue} and Lemma \ref{lemma eigenvalue} that if $(\delta_1,\delta_2)\in\sigma_p(T_1,T_2)$ and $x\in H$ such that $(T_i-\delta_iI)x=0$ for $i=1,2$, then $(f(T_1,T_2)-f(\delta_1,\delta_2)I)x=0$, that is $N(T-\delta I)\subseteq N\left(f(T)-f(\delta)I\right)$.
\end{note}
Next, we study the properties of image of non isolated point of $\sigma_T(T)$ under every function $f\in H(\sigma_T(T))$.
\begin{proposition}\label{prop not isolated}
	Let $T=(T_1,T_2)$ be a commuting pair of paranormal operators in $\mathcal{B}(H)$. If $\delta\in\sigma_T(T)$ is not an isolated point of $\sigma_T(T)$ then  one of the following holds.
	\begin{enumerate}
		\item $f(\delta)$ is not an isolated point of $\sigma_T(f(T))$.
		\item $f(\delta)$ is a joint eigenvalue of $f(T)$ with infinite multiplicity.
	\end{enumerate}
\end{proposition}
\begin{proof}
	Let $\delta=(\delta_1,\delta_2)$. We assume that $f(\delta)$ is an isolated point of $\sigma_T(f(T))$. The proof is divided into two cases.
	
	Case $(1)$: Both $\delta_1,\delta_2$ are not isolated points of $\sigma_T(T_1)$ and $\sigma_T(T_2)$, respectively. Then there exist two sequences $(\delta_{n}^1)\subseteq\sigma_T(T_1)$ and $(\delta_{n}^2)\subseteq\sigma_T(T_2)$ converging to $\delta_1$ and $\delta_2$, respectively. Consider the function $f_1(z)=f(z,\delta_2)$ defined in a neighbourhood of $\sigma_T(T_1)$. We know that the sequence $(f_1(\delta_{n_1}))\subseteq\sigma_T(f(T))$ converges to $f(\delta)$, but $f(\delta)$ is isolated point of $\sigma_T(f(T))$, this implies $f_1(\delta_{n}^1)=f(\delta),\,\forall n\in\mathbb{N}$. By \cite[Theorem 3.7, Page 78]{CONWAY}, we conclude that $f_1(z)=f(\delta)$ for all $z$ in a neighbourhood of $\delta_1$, say $U_1$.
	
	Now for fixed $z_1\in U_1$, we consider ${f_2}^{z_1}(w)=f(z_1,w)$ defined in a neighbourhood of $\sigma_T(T_2)$. By a similar argument as above, we see that ${f_2}^{z_1}(w)=f(\delta)$ for all $w$ in a neighbourhood of $\delta_2$, say $U_2$. From this, we conclude that $f(z,w)=f(\delta)$ for all $(z,w)\in U_1\times U_2$.
	
	If $\sigma_T(T)$ is connected, then $f(z,w)=f(\delta),\,\forall(z,w)\in\sigma_T(T)$ and consequently $f(T_1,T_2)=f(\delta)I$. Hence $f(\delta)$ is a joint eigenvalue with infinite multiplicity.
	
	If $\sigma_T(T)$ is disconnected, then $\sigma_T(T)$ is disjoint union of two compact sets $K_1,\,K_2$, where $K_1,\,K_2$ are connected components of $\sigma_T(T)$ and $(\delta_1,\delta_2)\in K_1$ (say). Thus $f(z,w)=f(\delta)$ for all $(z,w)\in K_1$. By Theorem \ref{thm disjoint spectra}, there exist a $T$-invariant subspace $M_1$ such that $\sigma_T(T|_{M_1})=K_1$ and consequently $f(T|_{M_1})=f(\delta)I_{M_1}$. We observe that $M_1$ is of infinite dimension, because if $M_1$ is finite dimensional, then $K_1$ is a finite set and consequently $(\delta_1,\delta_2)$ is isolated point of $\sigma_T(T)$, which is a contradiction to our assumption. Therefore $f(\delta)$ is a joint eigenvalue of $f(T)$ with infinite multiplicity.
	
	Case $(2)$: $\delta_1$ is not isolated point of $\sigma_T(T_1)$ and $\delta_2$ is an isolated point of $\sigma_T(T_2)$. By a similar argument as in Case $(1)$, we have $f(z_1,\delta_2)=f(\delta)$ for all $z_1\in V_1$, where $V_1$ is a neighbourhood of $\delta_1$. By \cite[Lemma 3.4]{UCHIYAMAp}, $\delta_2$ is an eigenvalue of $T_2$.
	
	For fixed $z_1\in V_1$, we get
	\begin{equation}\label{eqn roots}
	f(z_1,w)-f(z_1,\delta_2)=(w-\delta_2)^{m}h(z_1,w)
	\end{equation}
	for $w$ in a neighbourhood, say $V_2$ of $\delta_2$ and for some $m\in\mathbb{N}$, where $h(z_1,w)$ is analytic function in $V_2$ such that $h(z_1,\delta_2)\ne 0$. As $f(z_1,\delta_2)=f(\delta_1,\delta_2)$ for all $z_1\in V_1$, we get that Equation \ref{eqn roots} is equivalent to the following.
	\begin{equation}\label{equation of function}
	f(z_1,w)-f(\delta_1,\delta_2)=(w-\delta_2)^{m}h(z_1,w),\,\forall z_1\in V_1,\,w\in V_2.
	\end{equation}
	Now, we denote $N_1=N(T_2-\delta_2I)$. Then
	\begin{equation*}
	f(T_1|_{N_1},T_2|_{N_1})-f(\delta_1,\delta_2)I_{N_1}=h(T_1|_{N_1},T_2|_{N_1})(T_2|_{N_1}-\delta_2I_{N_1})^m.
	\end{equation*}
	In the above Equation $f(T_1|_{N_1},T_2|_{N_1})$ is well defined, because if $\sigma_T(T)$ is connected then Equation \ref{equation of function} holds true for every $(z_1,w)\in\sigma_T(T)$ and on the other hand if $\sigma(T)$ is disconnected, then $(\delta_1,\delta_2)$ belongs to some connected component of $\sigma_T(T)$, say $K_1$. Then Equation \ref{equation of function} holds for every $(z_1,w)\in K_1$. By theorem \ref{thm disjoint spectra}, if $M_1\subseteq H$ is the $T$-invariant subspace satisfying $\sigma_T(T_1|_{M_1},T_2|_{M_1})=K_1$, then $N\subseteq M_1$. Hence above equation is well defined.

	If $\delta_2$ is an eigenvalue of $T_2$ with infinite multiplicity, then $f(\delta)$ is a joint eigenvalue of $f(T)$ with infinite multiplicity, by the above equation.
	
	Now we claim that $\delta_2$ cannot be an eigenvalue with finite multiplicity. On the contrary assume that $\delta_2$ is an eigenvalue with finite multiplicity. Then \cite[Theorem 3.5]{UCHIYAMAp} implies that $T_2-\delta_2I$ is a Fredholm operator of index zero. Using \cite[(i), Page 70]{CURTO}, we see that $(T_1-z_1, T_2-\delta_2)$ is Taylor Fredholm of index zero, for every $z_1\in V_1$.
	
	By Theorem \ref{thm dash} either we have a sequence $(z_n)\subseteq V_1$ such that $(z_n,\delta_2)\subseteq\sigma_p(T)$ or we have a sequence $(\tilde{z}_n)\subseteq V_1$ such that $(\tilde{z}_n,\delta_2)\subseteq\sigma_p(T^*)^*$. Both the cases are not possible, because in the first case $\delta_2$ becomes an eigenvalue of $T_2$ with infinite multiplicity and in the second case $\bar{\delta}_2$ becomes an eigenvalue of $T_2^*$ with infinite multiplicity. Both contradicts the fact that $T_2-\delta_2I$ is a Fredholm operator.
\end{proof}
The following result can be viewed as a spectral mapping theorem for $\sigma_T(T)\setminus\pi_{00}(T)$.
\begin{proposition}\label{prop removing iso pts}
	Let $T=(T_1,T_2)$ be a commuting pair of paranormal operators in $\mathcal{B}(H)$. For every $f\in H(\sigma_T(T))$, we have $$\sigma_T(f(T))\setminus\pi_{00}(f(T))=f(\sigma_T(T)\setminus\pi_{00}(T)).$$
\end{proposition}
\begin{proof}
	We first prove that $f(\sigma_T(T)\setminus\pi_{00}(T))\subseteq\sigma_T(f(T))\setminus\pi_{00}(f(T))$. To see this, let $\lambda\in f(\sigma_T(T)\setminus\pi_{00}(T))$. For some $\delta=(\delta_1,\delta_2)\in\sigma_T(T)\setminus\pi_{00}(T)$, we have $f(\delta)=\lambda$.
	
	If $\delta$ is isolated point of $\sigma_T(T)$, then by Lemma \ref{Lemma *isoloid}, $\delta$ is a joint eigenvalue of $T$ and consequently $\delta$ is of infinite multiplicity. By Remark \ref{rem eigenvalues}, $\lambda$ is a joint eigenvalue of $f(T)$ with infinite multiplicity. Hence $\lambda\in\sigma_T(f(T))\setminus\pi_{00}(f(T))$.
	
	On the other hand, if $\delta$ is not isolated point of $\sigma_T(T)$, then by Proposition \ref{prop not isolated}, it follows that $\lambda\in\sigma_T(f(T))\setminus\pi_{00}(f(T))$.
	
	Our next objective is to show that $\sigma_T(f(T))\setminus\pi_{00}(f(T))\subseteq f(\sigma_T(T)\setminus\pi_{00}(T)).$ Let $\lambda\in\sigma_T(f(T))\setminus\pi_{00}(f(T))$. We prove the result by the following cases.
	
	Case $(1)$: $\lambda$ is not isolated point of $\sigma_T(f(T))$:\\
	 Then there exist a sequence $(\lambda_n)\subseteq\sigma_T(f(T))$ such that $(\lambda_n)$ converges to $\lambda$. As $(\lambda_n)\subseteq\sigma_T(f(T))=f(\sigma_T(T))$, we get $(\lambda_n)=(f(\delta_n))$ for some $(\delta_n)\subseteq\sigma_T(T)$. Using compactness of $\sigma_T(T)$, we obtain a subsequence $(\delta_{n_k})$ of $(\delta_n)$ converging to some $\delta_0\in\sigma_T(T)$. We know $(\lambda_{n_k})=(f(\delta_{n_k}))$ converges to $f(\delta_0)$, thus $\lambda=f(\delta_0)\in f(\sigma_T(T)\setminus\pi_{00}(T))$.
	
	Case $(2)$: $\lambda$ is an isolated point of $\sigma_T(f(T))$ but not a joint eigenvalue of $f(T)$:\\
	 We know that $\lambda\in \sigma_T(f(T))=f(\sigma_T(T))$, consequently $\lambda=f(\delta)$ for some $\delta\in \sigma_T(T)$. By Lemma \ref{iso eigenvalue}, if $\delta\in\pi_{00}(T)$, then $f(\delta)\in\pi_{00}(f(T))$, but $f(\delta)$ is not a joint eigenvalue of $f(T)$. As a consequence there exist a $\delta\in\sigma_T(T)\setminus\pi_{00}(T)$ such that $\lambda=f(\delta)$.
	
	Case $(3)$: $\lambda$ is an isolated joint eigenvalue of $f(T)$ with infinite multiplicity:\\
	 If $\lambda=f(\delta)$ for some $\delta\in\sigma_T(T)\setminus\pi_{00}(T)$, then we are done. Otherwise, consider the set $B:=\{\delta\in\pi_{00}(T):\lambda=f(\delta)\}$. We assert that $B$ is not a finite set.
	
	On the contrary, if the set is finite and equal to $\{\delta_1,\delta_2,\ldots\delta_{n_0}\}$ for some fixed $n_0\in\mathbb{N}$, then Lemma \ref{iso eigenvalue} implies that $\lambda$ is a joint eigenvalue of $f(T)$ with finite multiplicity, which is equal to the dimension of $\underset{i=1}{\overset{n_0}{\cup}}N(T-\delta_iI)$. This is a contradiction to our assumption that $\lambda$ is a joint eigenvalue of $f(T)$ with infinite multiplicity. This proves our assertion that $B$ is not finite. Thus the set $B$ has at least one limit point, say $\delta_0\in\sigma_T(T)$. Hence $\delta_0\in\sigma_T(T)\setminus\pi_{00}(T)$ and it satisfy $\lambda=f(\delta)$. This contradicts our assumption that every $\delta$ satisfying $\lambda=f(\delta)$ belongs to $\pi_{00}(T)$.
\end{proof}
Now we discuss the spectral mapping theorem for the joint Weyl spectrum.
\begin{theorem}\label{Weyl spectrum}
	Let $T=(T_1,T_2)$ be a commuting pair of paranormal operators in $\mathcal{B}(H)$. Then for every $f\in H(\sigma_T(T))$ we have $f(\omega(T))=\omega(f(T))$.
\end{theorem}
\begin{proof}
	Consider $\lambda\in\omega (f(T))$. We have to show that $\lambda\in f(\omega(T))$. By Theorem \ref{Weyls thm} and Proposition \ref{prop removing iso pts}, it is equivalent to show that $\lambda\in f(\sigma_T(T)\setminus\pi_{00}(T))=\sigma_T(f(T))\setminus\pi_{00}(f(T))$.
	
	On the contrary, we assume that $\lambda\in\pi_{00}(f(T))$. By Theorem \ref{thm disjoint spectra} there exist two subspaces $M,\,N\subseteq H$ invariant under $f(T)$ such that $\sigma_T(f(T)|_M)=\{\lambda\}$ and $\sigma_T(f(T)|_N)=\sigma_T(f(T))\setminus\{\lambda\}$. We know that $f(T)$ is similar to
\[
\begin{blockarray}{ccc}
M & N \\
\begin{block}{(cc)c}
 f(T)|_M& 0 &M\\
 0& B&N \\
\end{block}
\end{blockarray}, \text{ where }B\in\mathcal{B}(M^{\perp}).
\]
	Note that $M$ is finite dimensional, otherwise $\lambda\notin\pi_{00}(f(T))$.
	
	If $\lambda\ne 0$, then $f(T)-f(T)P_M$ is similar to
	\[
	\begin{bmatrix}
	0&0\\
	0&B
	\end{bmatrix}.
	\]
	It is easy to see that $\lambda\notin\sigma_T(f(T)-f(T)P_M)$.
	
	If $\lambda=0$, then $f(T)+P_M$ is similar to
	\[
	\begin{bmatrix}
	I_M&0\\
	0&B
	\end{bmatrix}.
	\]
	Thus $\lambda=0\notin\sigma_T(f(T+P_M))$.
	
	We know that $f(T)P_M$ and $P_M$ are finite rank operators, hence compact. This implies that $\lambda\notin\omega(f(T))$, which is a contradiction to our assumption.
	
	To get the reverse containment, let $\lambda\in f(\omega(T))$. Then there exist some $\delta\in\omega(T)$ such that $\lambda=f(\delta)$.
	
	If $\delta$ is an isolated point of $\sigma_T(T)$, then Lemma \ref{Lemma *isoloid} implies that $\delta$ is an isolated joint eigenvalue with infinite multiplicity. By Remark \ref{rem eigenvalues}, it follows that $\lambda$ is a joint eigenvalue of $f(T)$ with infinite multiplicity. From this we deduce that $\lambda\in\sigma_{T_W}(f(T))\subseteq \omega(f(T))$.
	
	If $\delta$ is not isolated point of $\sigma_T(T)$, then Proposition \ref{prop not isolated} shows that $f(T)-\lambda I$ is not Fredholm, hence $\lambda\in\sigma_{T_W}(f(T))\subseteq\omega(f(T))$.
\end{proof}
Summarizing the above results, we have the following Weyl's theorem for the analytic functional calculus for a commuting pair of paranormal operators.
\begin{theorem}
	Let $T=(T_1,T_2)$ be a commuting pair of paranormal operators in $\mathcal{B}(H)$. Then $f(T)$ satisfy Weyl's theorem-II for every $f\in H(\sigma_T(T))$. That is $$\sigma_T(f(T))\setminus\pi_{00}(f(T))=\omega(f(T)).$$
\end{theorem}
\begin{proof}
	By Proposition \ref{rem eigenvalues} and Theorem \ref{Weyl spectrum}, we have the following.
	\begin{align*}
	\sigma_T(f(T))\setminus\pi_{00}(f(T))=&f(\sigma_T(T)\setminus\pi_{00}(T))\\
	=&f(\omega(T))\\
	=&\omega(f(T)).
	\end{align*}
	This proves our result.
\end{proof}
\bibliographystyle{amsplain}

\end{document}